\documentclass[12pt]{article}

\usepackage{amsfonts}
\usepackage{graphicx}
\usepackage{amsmath}
\usepackage{amssymb}
\usepackage[numbers,square]{natbib}
\usepackage{url}
\usepackage{multirow}
\usepackage{color}
\usepackage{amsthm}

\newtheorem{theorem}{Theorem}
\newtheorem{corollary}[theorem]{Corollary}

\newtheorem{lemma}[theorem]{Lemma}
\newtheorem{proposition}[theorem]{Proposition}
\newtheorem{remark}[theorem]{Remark}

\numberwithin{theorem}{section}
\numberwithin{equation}{section}

\title{The Conway-Maxwell-Poisson distribution:\\ distributional theory and
approximation}
\author{Fraser Daly\footnote{Department of Actuarial Mathematics and Statistics and the Maxwell Institute for Mathematical Sciences, Heriot-Watt University, Edinburgh EH14 4AS, UK.  E-mail: f.daly@hw.ac.uk;  Tel: +44 (0)131 451 3212; Fax: +44 (0)131 451 3249}\; and Robert E$.$ Gaunt\footnote{Department of Statistics, University of Oxford, 24-29 St$.$ Giles', Oxford OX1 3LB, UK. E-mail: gaunt@stats.ox.ac.uk; Tel: +44 (0)1865 281279; Fax: +44 (0)1865 281333. }}
\date{May 2016}

\begin{document}

\maketitle

\noindent{\bf Abstract} The Conway-Maxwell-Poisson (CMP) distribution is a natural two-parameter generalisation of the Poisson distribution which has received some attention in the statistics literature in recent years by offering flexible generalisations of some well-known models.  In this work, we begin by establishing some properties of both the CMP distribution and an analogous generalisation of the binomial distribution, which we refer to as the CMB distribution.  We also consider some convergence results and approximations, including a bound on the total variation distance between a CMB distribution and the corresponding CMP limit.
\vspace{12pt}


\vspace{12pt}


\section{Introduction}\label{sec:intro}

A two-parameter generalisation of the Poisson distribution was introduced by Conway and Maxwell \cite{cm62} as the stationary number of occupants of a queuing system with state dependent service or arrival rates.  This distribution has since become known as the Conway-Maxwell-Poisson (CMP) distribution.  Beginning with the work of Boatwright, Borle and Kadane \cite{bbk03} and Shmueli et al$.$ \cite{s05}, the CMP distribution has received recent attention in the statistics literature on account of the flexibility it offers in statistical models.  For example, the CMP distribution can model data which is either under- or over-dispersed relative to the Poisson distribution.  This property is exploited by Sellers and Shmueli \cite{ss10}, who use the CMP distribution to generalise the Poisson and logistic regression models.  Kadane et al$.$ \cite{k06} considered the use of the CMP distribution in Bayesian analysis, and Wu, Holan and Wilkie \cite{whw13} use the CMP distribution as part of a Bayesian model for spatio-temporal data.  The CMP distribution is employed in a flexible cure rate model formulated by Rodrigues et al$.$ \cite{r09} and further analysed by Balakrishnan and Pal \cite{bp12}.  

Our purpose in this work is twofold.  Motivated by the use of the CMP distribution in the statistical literature, we firstly aim (in Section \ref{sec:prop}) to derive explicit distributional properties of the CMP distribution and an analogous generalisation of the binomial distribution, the CMB distribution.  Our second aim is to consider the CMP distribution as a limiting distribution.  We give conditions under which sums of dependent Bernoulli random variables will converge in distribution to a CMP random variable, and give an explicit bound in total variation distance between the CMB distribution and the corresponding CMP limit.  These convergence results are detailed in Sections \ref{sec:stein} and \ref{sec:conv}.

We use the remainder of this section to introduce the CMP and CMB distributions and collect some straightforward properties which will prove useful in the sequel.  We also introduce some further definitions that we will need in the work that follows.

\subsection{The CMP distribution}\label{subsec:cmp_def}

The CMP distribution is a natural two-parameter generalisation of the Poisson distribution.  We will write $X\sim\mbox{CMP}(\lambda,\nu)$ if
\begin{equation}\label{admpdf}
\mathbb{P}(X=j)=\frac{1}{Z(\lambda,\nu)}\frac{\lambda^j}{(j!)^\nu}\,,\qquad j\in\mathbb{Z}^+=\{0,1,2,\ldots\}\,,
\end{equation}   
where $Z(\lambda,\nu)$ is a normalizing constant defined by
$$
Z(\lambda,\nu)=\sum_{i=0}^\infty\frac{\lambda^i}{(i!)^\nu}\,.
$$
The domain of admissible parameters for which (\ref{admpdf}) defines a probability distribution is $\lambda,\nu>0$, and $0<\lambda<1$, $\nu=0$.

The introduction of the second parameter $\nu$ allows for either sub- or super-linear growth of the ratio $\mathbb{P}(X=j-1)/\mathbb{P}(X=j)$, and allows $X$ to have variance either less than or greater than its mean.  Of course, the mean of $X\sim\mbox{CMP}(\lambda,\nu)$ is not, in general, $\lambda$.  In Section \ref{sec:prop} we will consider further distributional properties of the CMP distribution, including expressions for its moments.

Clearly, in the case where $\nu=1$, $X\sim\mbox{CMP}(\lambda,1)$ has the Poisson distribution $\mbox{Po}(\lambda)$ and the normalizing constant $Z(\lambda,1)=e^\lambda$.  As noted by Shmueli et al$.$ \cite{s05}, other choices of $\nu$ also give rise to well-known distributions.  For example, in the case where $\nu=0$ and $0<\lambda<1$, $X$ has a geometric distribution, with $Z(\lambda,0)=(1-\lambda)^{-1}$.  In the limit $\nu\rightarrow\infty$, $X$ converges in distribution to a Bernoulli random variable with mean $\lambda(1+\lambda)^{-1}$ and $\lim_{\nu\rightarrow\infty}Z(\lambda,\nu)=1+\lambda$.

In general, of course, the normalizing constant $Z(\lambda,\nu)$ does not permit such a neat, closed-form expression.  Asymptotic results are available, however.  Gillispie and Green \cite{gg14} prove that, for fixed $\nu$,
\begin{equation}\label{eq:norm}
Z(\lambda,\nu)\sim\frac{\exp\left\{\nu\lambda^{1/\nu}\right\}}{\lambda^{(\nu-1)/2\nu}(2\pi)^{(\nu-1)/2}\sqrt{\nu}}\left(1+O\left(\lambda^{-1/\nu}\right)\right)\,,
\end{equation}
as $\lambda\rightarrow\infty$, confirming a conjecture made by Shmueli et al$.$ \cite{s05}.  This asymptotic result may also be used to obtain asymptotic results for the probability generating function of $X\sim\mbox{CMP}(\lambda,\nu)$, since it may be easily seen that
\begin{equation}\label{eq:pgf}
\mathbb{E}s^X=\frac{Z(s\lambda,\nu)}{Z(\lambda,\nu)}\,.
\end{equation}  

\subsection{The CMB distribution}

Just as the CMP distribution arises naturally as a generalisation of the Poisson distribution, we may define an analogous generalisation of the binomial distribution.  We refer to this as the Conway-Maxwell-binomial (CMB) distribution and write that $Y\sim\mbox{CMB}(n,p,\nu)$ if 
$$
\mathbb{P}(Y=j)=\frac{1}{C_n}\binom{n}{j}^\nu p^j(1-p)^{n-j}\,,\qquad j\in\{0,1,\ldots,n\}\,,
$$
where $n\in\mathbb{N}=\{1,2,\ldots\}$, $0\leq p\leq1$ and $\nu\geq0$.  The normalizing constant $C_n$ is defined by
$$
C_n=\sum_{i=0}^n\binom{n}{i}^\nu p^i(1-p)^{n-i}\,.
$$
The dependence of $C_n$ on $p$ and $\nu$ is suppressed for notational convenience.  Of course, the case $\nu=1$ is the usual binomial distribution $Y\sim\mbox{Bin}(n,p)$, with normalizing constant $C_n=1$.  Shmueli et al$.$ \cite{s05} considered the CMB distribution and derived some of its basic properties, referring to it as the CMP-binomial distribution.  We, however, consider it more natural to refer to this as the CMB distribution (a similar convention is also followed by Kadane \cite{k16}); we shall also later refer to an analogous generalisation of the Poisson binomial distribution as the CMPB distribution.   

There is a simple relationship between CMP and CMB random variables, which generalises a well-known result concerning Poisson and binomial random variables.  If $X_1\sim \mbox{CMP}(\lambda_1,\nu)$ and $X_2\sim \mbox{CMP}(\lambda_2,\nu)$ are independent, then $X_1\,|\,X_1+X_2=n\sim\mbox{CMB}(n,\lambda_1/(\lambda_1+\lambda_2),\nu)$ (see \cite{s05}).

It was also noted by \cite{s05} that $Y\sim\mbox{CMB}(n,p,\nu)$ may be written as a sum of exchangeable Bernoulli random variables $Z_1,\ldots,Z_n$ satisfying
\begin{equation}\label{eq:ex}
\mathbb{P}(Z_1=z_1,\ldots,Z_n=z_n)=\frac{1}{C_n}\binom{n}{k}^{\nu-1}p^k(1-p)^{n-k}\,,
\end{equation}
where $k=z_1+\cdots+z_n$.  Note that $\mathbb{E}Z_1\not=p$ in general, unless $\nu=1$.  However, $\mathbb{E}Z_1=n^{-1}\mathbb{E}Y$ may be either calculated explicitly or estimated using some of the properties of the CMB distribution to be discussed in the sequel.

From the mass functions given above, it can be seen that if $Y\sim\mbox{CMB}(n,\lambda/n^\nu,\nu)$, then $Y$ converges in distribution to $X\sim\mbox{CMP}(\lambda,\nu)$ as $n\rightarrow\infty$.  We return to this convergence in Section \ref{sec:stein} below, where we give an explicit bound on the total variation distance between these distributions.

\subsection{Power-biasing}\label{subsec:powerbias}

In what follows, we will need the definition of power-biasing, as used by Pek\"oz, R\"ollin and Ross \cite{prr14}.  For any non-negative random variable $W$ with finite $\nu$-th moment, we say that $W^{(\nu)}$ has the $\nu$-power-biased distribution of $W$ if
\begin{equation}\label{eq:powerbias}
(\mathbb{E}W^\nu)\mathbb{E}f(W^{(\nu)})=\mathbb{E}\left[W^\nu f(W)\right]\,,
\end{equation}
for all $f:\mathbb{R}^+\mapsto\mathbb{R}$ such that the expectations exist.  In this paper, we will be interested in the case that $W$ is non-negative and integer-valued.  In this case, the mass function of $W^{(\nu)}$ is given by
$$
\mathbb{P}(W^{(\nu)}=j)=\frac{j^\nu\mathbb{P}(W=j)}{\mathbb{E}W^\nu}\,,\qquad j\in\mathbb{Z}^+\,.
$$
Properties of a large family of such transformations, of which power-biasing is a part, are discussed by Goldstein and Reinert \cite{gr05}.  The case $\nu=1$ is the usual size-biasing, which has often previously been employed in conjunction with the Poisson distribution: see Barbour, Holst and Janson \cite{bhj92}, Daly, Lef\`evre and Utev \cite{dlu12}, Daly and Johnson \cite{dj13}, and references therein for some examples.  The power-biasing we employ here is the natural generalisation of size-biasing that may be applied in the CMP case.   

\section{Distributional properties of the CMP and CMB distributions} \label{sec:prop}

In this section we collect some distributional properties of the CMP and CMB distributions.  Some will be required in the sequel when considering approximations and convergence to the CMP distribution, and all of are of some interest, either independently or for statistical applications.

\subsection{Moments, cumulants, and related results}

We begin this section by noting, in Proposition \ref{prop:moments} below, that some moments of the CMP distribution may be easily and explicitly calculated. The simple formula $\mathbb{E}X^\nu=\lambda$ was already known to Sellers and Shmueli \cite{ss10}.  We also note the corresponding result for the CMB distribution.  

Here and in the sequel we let 
$$
(j)_r=j(j-1)\cdots(j-r+1)
$$
denote the falling factorial.
\begin{proposition} \label{prop:moments}
\begin{enumerate}
\item[(i).] Let $X\sim\mbox{CMP}(\lambda,\nu)$, where $\lambda,\nu>0$.  Then
$$
\mathbb{E}[((X)_r)^\nu]=\lambda^r\,,
$$
for $r\in\mathbb{N}$.
\item[(ii).] Let $Y\sim\mbox{CMB}(n,p,\nu)$, where $\nu>0$.  Then
$$
\mathbb{E}[((Y)_r)^\nu]=\frac{C_{n-r}}{C_n}((n)_r)^\nu p^r\,,
$$
for $r=1,\ldots,n-1$.
\end{enumerate}
\end{proposition}

\begin{proof}We have
\begin{align*}\mathbb{E}[((X)_r)^\nu]&=\frac{1}{Z(\lambda,\nu)}\sum_{k=0}^\infty ((k)_r)^\nu\frac{\lambda^k}{(k!)^\nu}=\frac{\lambda^r}{Z(\lambda,\nu)}\sum_{k=r}^\infty \frac{\lambda^{k-r}}{((k-r)!)^\nu}\\&=\frac{\lambda^r}{Z(\lambda,\nu)}\sum_{j=0}^\infty \frac{\lambda^{j}}{(j!)^\nu}=\lambda^r\, ,
\end{align*}
and
\begin{align*}\mathbb{E}[((Y)_r)^\nu]&=\frac{1}{C_n}\sum_{k=0}^n((k)_r)^\nu \binom{n}{k}^\nu p^k(1-p)^{n-k}\\
&=\frac{1}{C_n}\bigg(\frac{n!}{(n-r)!}\bigg)^\nu \sum_{k=r}^n\binom{n-r}{k-r}^\nu p^k(1-p)^{n-k} \\
&=\frac{1}{C_n}((n)_r)^\nu p^r \sum_{j=0}^{n-r}\binom{n-r}{j}^\nu p^j(1-p)^{n-r-j}=\frac{C_{n-r}}{C_n}((n)_r)^\nu p^r\,.
\end{align*}
\end{proof}

\begin{remark}\emph{It is well-known that the factorial moments of $Z\sim\mathrm{Po}(\lambda)$ are given by $\mathbb{E}[(X)_r]=\lambda^r$.  We therefore have the attractive formula $\mathbb{E}[((X)_r)^\nu]=\mathbb{E}[(Z)_r]$, for $X\sim\mbox{CMP}(\lambda,\nu)$.}
\end{remark}

Such simple expressions do not exist for moments of $X\sim\mbox{CMP}(\lambda,\nu)$ which are not of the form $\mathbb{E}[((X)_r)^\nu]$.  Instead, we use (\ref{eq:norm}) to give asymptotic expressions for such moments. 

\begin{proposition}\label{prop:moments2}
Let $X\sim\mbox{CMP}(\lambda,\nu)$.  Then, for $k\in\mathbb{N}$,
$$
\mathbb{E}X^k\sim\lambda^{k/\nu}\left(1+O\left(\lambda^{-1/\nu}\right)\right)\,,
$$
as $\lambda\rightarrow\infty$.
\end{proposition}
\begin{proof}
It is clear that, for $k\in\mathbb{N}$,
\begin{equation*}
\mathbb{E}[(X)_k]=\frac{\lambda^k}{Z(\lambda,\nu)}\frac{\partial^k}{\partial \lambda^k}Z(\lambda,\nu)\,.
\end{equation*}
Differentiating (\ref{eq:norm}) (see Remark \ref{difasy} for a justification) we have that 
\begin{equation}\label{ggexpan}
\frac{\partial^k}{\partial \lambda^k}Z(\lambda,\nu)\sim \lambda^{k/\nu-k}\cdot\frac{\exp\left\{\nu\lambda^{1/\nu}\right\}}{\lambda^{(\nu-1)/2\nu}(2\pi)^{(\nu-1)/2}\sqrt{\nu}}\left(1+O\left(\lambda^{-1/\nu}\right)\right)\,,
\end{equation}
as $\lambda\rightarrow\infty$, and hence
$$
\mathbb{E}[(X)_k]\sim\lambda^{k/\nu}\left(1+O\left(\lambda^{-1/\nu}\right)\right)\,,
$$
as $\lambda\rightarrow\infty$.  We now exploit the following connection between moments and factorial moments:
\begin{equation}\label{eq:stirling}
\mathbb{E}X^k=\sum_{r=1}^k{k\brace r}\mathbb{E}[(X)_r]\,, 
\end{equation}
for $k\in\mathbb{N}$, where the Stirling numbers of the second kind ${k\brace r}$ are given by ${k\brace r}=\frac{1}{r!}\sum_{j=0}^r(-1)^{r-j}\binom{r}{j}j^k$ (see Olver et al$.$ \cite{o10}).  Using (\ref{eq:stirling}), and noting that ${k\brace k}=1$, completes the proof.
\end{proof}

\begin{remark}\label{difasy}\emph{In the above proof, we differentiated the asymptotic formula (\ref{eq:norm}) in the naive sense by simply differentiating the leading term $k$ times.  We shall also do this below in deriving the variance formula (\ref{eq:var}), and in Proposition \ref{kappaprop}, in which we differentiate an asymptotic series for $\log(Z(\lambda e^{t},\nu))$ with respect to $t$ in an analogous manner.  However, as noted by Hinch \cite{h91}, p$.$ 23, asymptotic approximations cannot be differentiated in this manner in general.  Fortunately, in the case of the asymptotic expansion (\ref{eq:norm}) for $Z(\lambda,\nu)$ we can do so.  This is because we have the following asymptotic formula for the CMP normalising constant that is more precise than (\ref{eq:norm}).  For fixed $\nu$,
\begin{equation}\label{eq:norm1}
Z(\lambda,\nu)\sim\frac{\exp\left\{\nu\lambda^{1/\nu}\right\}}{\lambda^{(\nu-1)/2\nu}(2\pi)^{(\nu-1)/2}\sqrt{\nu}}\bigg(1+\sum_{k=1}^\infty a_k\lambda^{-k/\nu}\bigg)\,,
\end{equation}
as $\lambda\rightarrow\infty$, where the $a_k$ are constants that do not involve $\lambda$.  The $m$-th derivative of the asymptotic series (\ref{eq:norm1}) is dominated by the $m$-th derivative of the leading term of (\ref{eq:norm1}), meaning that one can naively differentiate the asymptotic series, as we did in the proof of Proposition \ref{prop:moments2}. } 

\emph{The leading term in the asymptotic expansion (\ref{eq:norm1}) was obtained for integer $\nu$ by Shmueli et al$.$ \cite{s05}, and then for all $\nu>0$ by Gillispie and Green \cite{gg14}.  When stating their results, \cite{s05} and \cite{gg14} did not include the lower order term $\sum_{k=1}^\infty a_k\lambda^{-k/\nu}$, but it can be easily read off from their analysis.  For integer $\nu$, \cite{s05} gave an integral representation for $Z(\lambda,\nu)$ and then applied Laplace's approximation to write down the leading order term in its asymptotic expansion.  Laplace's approximation gives that (see Shun and McCullagh \cite{sm95}, p$.$ 750), for infinitely differentiable $g:\mathbb{R}^d\rightarrow\mathbb{R}$, 
\begin{equation*}\int_{\mathbb{R}^d}\exp\{-ng(x)\}\,dx\sim\bigg(\frac{n\det(\hat{g}'')}{2\pi}\bigg)^{-1/2}\exp\{-n\hat{g}\}\bigg(1+\sum_{k=1}^\infty b_k n^{-k}\bigg)\,, \quad \text{as } n\rightarrow\infty\,,
\end{equation*} 
where the $b_k$ do not involve $n$, and $\hat{g}$ and $\hat{g}''$ denote $g$ and the matrix of second order derivatives of $g$, respectively, evaluated at the value $\hat{x}$ that minimises $g$.  It is now clear that the lower order term in (\ref{eq:norm1}) has the form $\sum_{k=1}^\infty a_k\lambda^{-k/\nu}$.  For general $\nu>0$, \cite{gg14} obtained an expression for the leading term in the asymptotic expansion of $Z(\lambda,\nu)$ by using Laplace's approximation, as well as several other simpler asymptotic approximations.  In each of these approximations, the lower order term is of the form $\sum_{k=1}^\infty c_k\lambda^{-k/\nu}$, from which it follows that $Z(\lambda,\nu)$ has an asymptotic expansion of the form (\ref{eq:norm1}). }
\end{remark}

We also have the following relationship between moments of $X\sim\mbox{CMP}(\lambda,\nu)$:
$$
\mathbb{E}X^{r+1}=\lambda\frac{d}{d\,\lambda}\mathbb{E}X^r+\mathbb{E}X\mathbb{E}X^r\,,
$$
for $r>0$.  See equation (6) of Shmueli et al$.$ \cite{s05}.  With $r=1$ we obtain 
\begin{equation}\label{eq:var}
\mbox{Var}(X)=\lambda\frac{d}{d\,\lambda}\mathbb{E}X\sim\frac{1}{\nu}\lambda^{1/\nu}+O(1)\,,
\end{equation}  
as $\lambda\rightarrow\infty$, from Proposition \ref{prop:moments2}.  This also gives the following corollary.
\begin{corollary}
Let $m$ be the median of $X\sim\mbox{CMP}(\lambda,\nu)$.  Then
$$
m\sim\lambda^{1/\nu}+O\left(\lambda^{1/2\nu}\right)\,,
$$
as $\lambda\rightarrow\infty$.
\end{corollary} 
\begin{proof}
From above, $\mathbb{E}X\sim\lambda^{1/\nu}+O(1)$ and $\sigma=\sqrt{\mbox{Var}(X)}\sim\frac{1}{\sqrt{\nu}}\lambda^{1/2\nu}+O\left(1\right)$.  Since $\sigma<\infty$, we may use a result of Mallows \cite{m91}, who showed that $|\mathbb{E}X-m|\leq\sigma$.  The result follows.
\end{proof}

As with the moments above, we may also find asymptotic expressions for the cumulants of the CMP distribution. 
\begin{proposition}\label{kappaprop}
For $n\geq1$, let $\kappa_n$ be the $n$th cumulant of $X\sim\mbox{CMP}(\lambda,\nu)$.  Then
$$
\kappa_n\sim\frac{1}{\nu^{n-1}}\lambda^{1/\nu}+O(1)\,,  
$$
as $\lambda\rightarrow\infty$.
\end{proposition}
\begin{proof}
From (\ref{eq:pgf}), the cumulant generating function of $X\sim\mbox{CMP}(\lambda,\nu)$ is
$$
g(t)=\log(\mathbb{E}[e^{tX}])=\log(Z(\lambda e^{t},\nu))-\log(Z(\lambda,\nu)).
$$
The cumulants are given by
$$
\kappa_n=g^{(n)}(0)=\frac{\partial^n}{\partial t^n}\log(Z(\lambda e^{t},\nu))\bigg|_{t=0}.
$$
From (\ref{eq:norm}),
$$
\log(Z(\lambda e^{t},\nu))\sim \nu\lambda^{1/\nu}e^{t/\nu}\,, 
$$
as $\lambda\rightarrow\infty$.  The expression for the leading term in the asymptotic expansion of $\kappa_n$ now easily follows, and a straightforward analysis, which is omitted, shows that the second term is $O(1)$ for all $n\geq1$.  The result now follows.
\end{proof}

Note that as a corollary to this result, the skewness $\gamma_1$ of $X\sim\mbox{CMP}(\lambda,\nu)$ satisfies
$$
\gamma_1=\frac{\kappa_3}{\sigma^3}\sim\frac{1}{\sqrt{\nu}}\lambda^{-1/2\nu}+O\left(\lambda^{-3/2\nu}\right)\,,
$$
as $\lambda\rightarrow\infty$, where $\sigma^2=\mbox{Var}(X)\sim\frac{1}{\nu}\lambda^{1/\nu}$ from (\ref{eq:var}).  Similarly, the excess kurtosis $\gamma_2$ of $X$ satisfies
$$
\gamma_2=\frac{\kappa_4}{\sigma^4}\sim\frac{1}{\nu}\lambda^{-1/\nu}+O\left(\lambda^{-2/\nu}\right)\,,
$$
as $\lambda\rightarrow\infty$.  For comparison, recall that in the Poisson case ($\nu=1$), $\gamma_1=\lambda^{-1/2}$ and $\gamma_2=\lambda^{-1}$.

We conclude this section with two further results of a similar flavour.  We begin by giving expressions for the modes of the CMP and CMB distributions. The expression for the mode of CMP distribution for non-integral $\lambda^{1/\nu}$ was known to Guikema and Goffelt \cite{gg08}, but for clarity and completeness we state the result and give the simple proof.  The expression for the mode of the CMB distribution is new.  Here and in the sequel we will let $\lfloor\cdot\rfloor$ denote the floor function.
\begin{proposition}
\begin{enumerate}
\item[(i).]Let $X\sim\mbox{CMP}(\lambda,\nu)$.  Then the mode of $X$ is $\lfloor\lambda^{1/\nu}\rfloor$ if $\lambda^{1/\nu}$ is not an integer.  Otherwise, the modes of $X$ are $\lambda^{1/\nu}$ and $\lambda^{1/\nu}-1$.
\item[(ii).]Let $Y\sim\mbox{CMB}(n,p,\nu)$ and define
$$
a=\frac{n+1}{1+\left(\frac{1-p}{p}\right)^{1/\nu}}\,.
$$
Then the mode of $Y$ is $\lfloor a\rfloor$ if $a$ is not an integer.  Otherwise, the modes of $Y$ are $a$ and $a-1$. 
\end{enumerate}
\end{proposition}
\begin{proof}
\begin{enumerate}
\item[(i).] Writing
$$
\mathbb{P}(X=j)=\frac{1}{Z(\lambda,\nu)}\left(\frac{(\lambda^{1/\nu})^j}{j!}\right)^\nu\,,
$$
the result now follows as in the Poisson case, for which the result is well-known.
\item[(ii).] This is a straightforward generalisation of the derivation of the mode of a binomial distribution given by Kaas and Buhrman \cite{kb80}.  Consideration of the ratio
$$
\frac{\mathbb{P}(Y=k+1)}{\mathbb{P}(Y=k)}=\left(\frac{n-k}{k+1}\right)^\nu \frac{p}{1-p}=\left(\frac{n-k}{k+1}\left(\frac{p}{1-p}\right)^{1/\nu}\right)^\nu\,,
$$
shows that $\mathbb{P}(Y=k)$ increases as a function of $k$ if $k<a$ and decreases for $k>a-1$.  Therefore, if $a$ is not an integer, $\mathbb{P}(Y=k)$ increases for $k\leq\lfloor{a}\rfloor$ and decreases for $k\geq\lfloor{a}\rfloor$, giving $\lfloor{a}\rfloor$ as the mode.  If $a$ is an integer then $\mathbb{P}(Y=k)$ increases for $k\leq a-1$ and decreases for $k\geq a$ and so $a-1$ and $a$ are neighbouring modes. 
\end{enumerate}
\end{proof}

We also give, in Proposition \ref{prop:md} below, an expression for the mean deviation of $X^\nu$, where $X\sim\mbox{CMP}(\lambda,\nu)$, as usual.  This generalises a result of Crow \cite{c58}, who showed that if $Z\sim\mbox{Po}(\lambda)$ has a Poisson distribution, then
$$
\mathbb{E}|Z-\lambda|=2e^{-\lambda}\frac{\lambda^{\lfloor\lambda\rfloor+1}}{\lfloor\lambda\rfloor!}=2\lambda S\,,
$$
where $S=e^{-\lambda}\frac{\lambda^{\lfloor\lambda\rfloor+1}}{\lfloor\lambda\rfloor!}$ is the maximum value of the mass function of $Z$.
\begin{proposition}\label{prop:md}
Let $X\sim\mbox{CMP}(\lambda,\nu)$.  Then
$$
\mathbb{E}|X^\nu-\lambda|=2Z(\lambda,\nu)^{-1}\frac{\lambda^{\lfloor\lambda^{1/\nu}\rfloor+1}}{\lfloor\lambda^{1/\nu}\rfloor!}=2\lambda T\,,
$$
where $T=Z(\lambda,\nu)^{-1}\frac{\lambda^{\lfloor\lambda^{1/\nu}\rfloor+1}}{\lfloor\lambda^{1/\nu}\rfloor!}$, the maximum value of the mass function of $X$.
\end{proposition} 
\begin{proof}
\begin{align*}\mathbb{E}|X^\nu-\lambda| &=\sum_{k=0}^{\infty}|k^\nu-\lambda|Z(\lambda,\nu)^{-1}\frac{\lambda^k}{(k!)^{\nu}} \\
&=Z(\lambda,\nu)^{-1}\bigg[\sum_{k=0}^{\lfloor{\lambda^{1/\nu}}\rfloor}(\lambda-k^\nu)\frac{\lambda^k}{(k!)^\nu}+\sum_{k=\lfloor{\lambda^{1/\nu}}\rfloor+1}^{\infty}(k^\nu-\lambda)\frac{\lambda^k}{(k!)^\nu}\bigg] \\
&=2Z(\lambda,\nu)^{-1}\frac{\lambda^{\lfloor{\lambda^{1/\nu}}\rfloor+1}}{\lfloor{\lambda^{1/\nu}}\rfloor!}\,.
\end{align*}
\end{proof}

\subsection{Characterisations}

In Section 3, we will use Stein's method for probability approximations (see, for example, Stein \cite{s72} or Chen \cite{c75}) to give bounds on the convergence of the CMB distribution to a suitable CMP limit.  Stein's method relies on linear operators characterising distributions of interest.  In the following lemma, we present such characterisations for the CMP and CMB distributions.  These will also prove useful in deriving several other properties of these distributions in the work that follows.
\begin{lemma}\label{lem:char}
We have the following characterisations for the CMP and CMB distributions.
\begin{enumerate}
\item[(i).]Let $X\sim\mbox{CMP}(\lambda,\nu)$, and suppose that $f:\mathbb{Z}^+\mapsto\mathbb{R}$ is such that $\mathbb{E}|f(X+1)|<\infty$ and $\mathbb{E}|X^{\nu}f(X)|<\infty$.  Then 
\begin{equation}\label{char1112}\mathbb{E}[\lambda f(X+1)-X^\nu f(X)]=0.
\end{equation}
Conversely, suppose now that $W$ is a real-valued random variable supported on $\mathbb{Z}^+$ such that $\mathbb{E}[\lambda f(W+1)-W^\nu f(W)]=0$ for all bounded $f:\mathbb{Z}^+\mapsto\mathbb{R}$.  Then $W\sim \mbox{CMP}(\lambda,\nu)$.
\item[(ii).]Let $Y\sim\mbox{CMB}(n,p,\nu)$,  and suppose that $f:\mathbb{Z}^+\mapsto\mathbb{R}$ is such that $\mathbb{E}|f(Y+1)|<\infty$ and $\mathbb{E}|Y^{\nu}f(Y)|<\infty$.  Then 
\begin{equation}\label{char1113}\mathbb{E}[p(n-Y)^\nu f(Y+1)-(1-p)Y^\nu f(Y)]=0.
\end{equation}
\end{enumerate}
\end{lemma}

\begin{proof}The characterising equations (\ref{char1112}) and (\ref{char1113}) may be obtained directly through straightforward manipulations, or from the work of Brown and Xia \cite{bx01}, for example, who consider such characterisations in the more general setting of the equilibrium distribution of a birth-death process.  To prove the converse statement for the CMP distribution in part (i), we consider the so-called Stein equation for the CMP distribution:
\begin{equation}\label{eq:stein111}
I(x\in A)-\mathbb{P}(X\in A)=\lambda f_A(x+1)-x^\nu f_A(x),
\end{equation}
where $X\sim\mbox{CMP}(\lambda,\nu)$, $A\subseteq \mathbb{Z}^+$ and $f_A:\mathbb{Z}^+\mapsto\mathbb{R}$.  In Lemma \ref{lem:stein} below, we will obtain the unique solution to (\ref{eq:stein111}) and prove that it is bounded.  Since $f_A$ is bounded, evaluating both sides of (\ref{eq:stein111}) at $W$ and taking expectations gives that, for any $A\subseteq \mathbb{Z}^+$,
\begin{equation*}
\mathbb{P}(W\in A)-\mathbb{P}(X\in A)=\mathbb{E}[\lambda f_A(W+1)-W^\nu f_A(W)]=0,
\end{equation*}
from which it follows that $W\sim \mbox{CMP}(\lambda,\nu)$.
\end{proof}

We do not give a converse statement, and thus a complete characterisation, for the CMB distribution.  We do not need the converse in this paper, and giving a proof analogous given to that for CMP distribution would be tedious, because we would need to solve the corresponding CMB Stein equation and then derive bounds for the solution.

\subsection{Stochastic ordering and related results}

In this section we will explore some properties of CMP and CMB distributions that may be obtained by considering various stochastic orderings.  In particular, we will use the usual stochastic order and the convex order.  For random variables $U$ and $V$, we say that $U$ is smaller than $V$ in the usual stochastic order (which we denote $U\leq_{st}V$) if $\mathbb{E}f(U)\leq\mathbb{E}f(V)$ for all increasing functions $f$.  Equivalently, $U\leq_{st}V$ if $\mathbb{P}(U>t)\leq\mathbb{P}(V>t)$ for all $t$.  For random variables $U$ and $V$ with $\mathbb{E}U=\mathbb{E}V$, we will say that $U$ is smaller than $V$ in the convex order (written $U\leq_{cx}V$) if $\mathbb{E}f(U)\leq\mathbb{E}f(V)$ for all convex functions $f$.  Many further details on these orderings may be found in the book by Shaked and Shanthikumar \cite{ss07}, for example.  

We begin with two lemmas that make use of the power biasing introduced in Section \ref{subsec:powerbias}.
\begin{lemma}\label{lem:ord1}
Let $W$ be a non-negative random variable and $0\leq\alpha<\beta$.   Suppose that $\mathbb{E}W^\alpha$ and $\mathbb{E}W^\beta$ exist.  Then $W^{(\alpha)}\leq_{st}W^{(\beta)}$.
\end{lemma}
\begin{proof}
It is easily checked, using the definition (\ref{eq:powerbias}), that for $\alpha,\delta\geq0$, $\left(W^{(\alpha)}\right)^{(\delta)}=_{st}W^{(\alpha+\delta)}$.  Taking $\delta=\beta-\alpha>0$, it therefore suffices to prove the lemma with $\alpha=0$.  That is, we need to show that $\mathbb{E}f(W)\leq\mathbb{E}f(W^{(\beta)})$ for all $\beta>0$ and all increasing functions $f$.  From the definition (\ref{eq:powerbias}) this is immediate, since $\mbox{Cov}(W^\beta,f(W))\geq0$ for all increasing $f$.
\end{proof}

\begin{lemma}\label{lem:ord2}
Let $Y\sim\mbox{CMB}(n,p,\nu)$.  Then $Y^{(\nu)}\leq_{st}Y+1$.
\end{lemma}
\begin{proof}
Note that
\begin{align*}
\mathbb{P}(Y+1=j)&=\frac{1}{C_n}\binom{n}{j-1}^\nu p^{j-1}(1-p)^{n-j+1}\,,\qquad j=1,\ldots,n+1\,;\\
\mathbb{P}(Y^{(\nu)}=j)&=\frac{1}{C_n\mathbb{E}Y^\nu}\frac{p(n-j+1)^\nu}{1-p}\binom{n}{j-1}^\nu p^{j-1}(1-p)^{n-j+1}\,,\qquad j=1,\ldots,n\,.
\end{align*}
Hence, the condition that $\mathbb{E}f(Y^{(\nu)})\leq\mathbb{E}f(Y+1)$ (for all increasing $f:\mathbb{Z}^+\mapsto\mathbb{R}$) is equivalent to the non-negativity of
\begin{equation}\label{eq:ord}
\mathbb{E}\left[f(Y+1)\left\{1-\frac{p(n-Y)^\nu}{(1-p)\mathbb{E}Y^\nu}\right\}\right]\,,
\end{equation}
for all $f$ increasing.  Noting that, by Lemma \ref{lem:char} (ii), $\mathbb{E}\left[1-\frac{p(n-Y)^\nu}{(1-p)\mathbb{E}Y^\nu}\right]=0$, (\ref{eq:ord}) is the covariance of two increasing functions, and is hence non-negative.
\end{proof}

\subsubsection{Ordering results for CMP distributions}

Throughout this section, let $X\sim\mbox{CMP}(\lambda,\nu)$.  It is clear that $X^{(\nu)}=_{st}X+1$ and hence, for $\nu\geq1$, Lemma \ref{lem:ord1} gives $X^{(1)}\leq_{st}X+1$.  This is the negative dependence condition employed by Daly, Lef\`evre and Utev \cite{dlu12}.  Some consequences of this stochastic ordering are given in Proposition \ref{prop:cmp_ord1} below.  Before we can state these, we define the total variation distance between non-negative, integer-valued random variables $U$ and $V$:
$$
d_{TV}(\mathcal{L}(U),\mathcal{L}(V))=\sup_{A\subseteq\mathbb{Z}^+}|\mathbb{P}(U\in A)-\mathbb{P}(V\in A)|\,.
$$
We will also need to define the Poincar\'e (inverse spectral gap) constant $R_U$ for a non-negative, integer-valued random variable $U$:
$$
R_U=\sup_{g\in\mathcal{G}(U)}\left\{\frac{\mathbb{E}[g(U)^2]}{\mathbb{E}[\left\{g(U+1)-g(U)\right\}^2]}\right\}\,,
$$
where the supremum is take over the set
$$
\mathcal{G}(U)=\left\{g:\mathbb{Z}^+\mapsto\mathbb{R}\mbox{ with }\mathbb{E}[g(U)^2]<\infty\mbox{ and }\mathbb{E}g(U)=0\right\}\,.
$$
\begin{proposition}\label{prop:cmp_ord1}
Let $X\sim\mbox{CMP}(\lambda,\nu)$ with $\nu\geq1$.  Let $\mu=\mathbb{E}X$.  Then
\begin{enumerate}
\item[(i).]
$$
d_{TV}(\mathcal{L}(X),\mbox{Po}(\mu))\leq\frac{1}{\mu}\left(\mu-\mbox{Var}(X)\right)\sim\frac{\nu-1}{\nu}+O\left(\lambda^{-1/\nu}\right)\,,
$$
as $\lambda\rightarrow\infty$.
\item[(ii).]
$$
\mbox{Var}(X)\leq R_X\leq\mu\,.
$$
\item[(iii).] $X\leq_{cx}Z$, where $Z\sim\mbox{Po}(\mu)$.  In particular,
\begin{align*}
\mathbb{P}(X\geq\mu+t)&\leq e^t\left(1+\frac{t}{\mu}\right)^{-(\mu+t)}\,,\\
\mathbb{P}(X\leq\mu-t)&\leq e^{-t}\left(1-\frac{t}{\mu}\right)^{t-\mu}\,,
\end{align*}
where the latter bound applies if $t<\mu$. 
\end{enumerate}
\end{proposition}
\begin{proof}
The upper bound in (i) follows from Proposition 3 of Daly, Lef\`evre and Utev \cite{dlu12}.  The asymptotic behaviour of the upper bound is a consequence of our Proposition \ref{prop:moments2} and (\ref{eq:var}).  In (ii), the lower bound is standard and the upper bound is from Theorem 1.1 of Daly and Johnson \cite{dj13}.  (iii) follows from Theorem 2.2 and Corollary 2.8 of Daly \cite{d15}.
\end{proof}

On the other hand, if $\nu<1$ we have that $X+1\leq_{st}X^{(\nu)}$.  In that case, Proposition 3 of Daly, Lef\`evre and Utev \cite{dlu12} gives the upper bound
$$
d_{TV}(\mathcal{L}(X),\mbox{Po}(\mu))\leq\frac{1}{\mu}\left(\mbox{Var}(X)-\mu\right)\sim\frac{1-\nu}{\nu}+O(\lambda^{-\/\nu})\,,
$$ 
as $\lambda\rightarrow\infty$, where $\mu=\mathbb{E}X$ and the asymptotics of the upper bound again follow from our Proposition \ref{prop:moments2} and (\ref{eq:var}).  

However, in the case $\nu<1$ we cannot adapt the proof of Theorem 1.1 of Daly and Johnson \cite{dj13} to give an analogue of Proposition \ref{prop:cmp_ord1} (ii). By suitably modifying the proof of Theorem 2.2 of Daly \cite{d15}, we may note that $Z\leq_{cx}X$ in this case, where $Z\sim\mbox{Po}(\mu)$.  There is, however, no concentration inequality corresponding to that given in Proposition \ref{prop:cmp_ord1} (iii). 

\subsubsection{Ordering results for CMB distributions}

Now let $Y\sim\mbox{CMB}(n,p,\nu)$.  In the case $\nu\geq1$ we may combine Lemmas \ref{lem:ord1} and \ref{lem:ord2} to see that $Y^{(1)}\leq_{st}Y+1$.  That is, the negative dependence condition holds.  We thus have the following analogue of Proposition \ref{prop:cmp_ord1}, which may be proved in the same way as that result.
\begin{proposition}
Let $Y\sim\mbox{CMB}(n,p,\nu)$ with $\nu\geq1$.  Let $\mu=\mathbb{E}Y$.  Then
\begin{enumerate}
\item[(i).]
$$
d_{TV}(\mathcal{L}(Y),\mbox{Po}(\mu))\leq\frac{1}{\mu}\left(\mu-\mbox{Var}(Y)\right)\,.
$$
\item[(ii).]
$$
\mbox{Var}(Y)\leq R_Y\leq\mu\,.
$$
\item[(iii).] $Y\leq_{cx}X$, where $X\sim\mbox{Po}(\mu)$.  In particular,
\begin{align*}
\mathbb{P}(Y\geq\mu+t)&\leq e^t\left(1+\frac{t}{\mu}\right)^{-(\mu+t)}\,,\\
\mathbb{P}(Y\leq\mu-t)&\leq e^{-t}\left(1-\frac{t}{\mu}\right)^{t-\mu}\,,
\end{align*}
where the latter bound applies if $t<\mu$. 
\end{enumerate}
\end{proposition}
There is no corresponding result in the case $\nu<1$, since the stochastic ordering $Y^{(\nu)}\leq_{st}Y+1$ holds regardless of the sign of $1-\nu$, and so we cannot use our previous lemmas to make a stochastic comparison between $Y^{(1)}$ and $Y+1$ when $Y\sim\mbox{CMB}(n,p,\nu)$ with $\nu<1$.

\subsection{Example: $\nu=2$}

For some illustration of our results, consider the case $\nu=2$ and let $X\sim\mbox{CMP}(\lambda,2)$.  Let $I_r(x)$ be the modified Bessel function of the first kind defined by
$$
I_r(x)=\sum_{k=0}^\infty\frac{1}{k!\Gamma(r+k+1)}\left(\frac{x}{2}\right)^{r+2k}\,.
$$
Note that, by definition, the normalizing constant $Z(\lambda,2)=I_0(2\sqrt{\lambda})$.  Hence, for $m\in\mathbb{N}$,
\begin{align*}
\mathbb{E}[(X)_m]&=\frac{1}{I_0(2\sqrt{\lambda})}\sum_{k=m}^{\infty}\frac{\lambda^k}{k!(k-m)!}\\
&=\frac{\lambda^{m/2}I_m(2\sqrt{\lambda})}{I_0(2\sqrt{\lambda})}\,.
\end{align*}
From (\ref{eq:stirling}) we therefore have
$$
\mathbb{E}X^m=\sum_{k=1}^m{m\brace k}\frac{\lambda^{k/2}I_k(2\sqrt{\lambda})}{I_0(2\sqrt{\lambda})}\,.
$$
In particular, the mean of $X$ is given by
$$
\mathbb{E}X=\frac{\sqrt{\lambda}I_1(2\sqrt{\lambda})}{I_0(2\sqrt{\lambda})}\,.
$$
Also, since $\mathbb{E}X^2=\lambda$, the variance is given by
$$
\mathrm{Var}(X)=\lambda\left(1-\frac{I_1(2\sqrt{\lambda})^2}{I_0(2\sqrt{\lambda})^2}\right)\,.
$$
Formulas for the cumulants, skewness and excess kurtosis of $X$ can also be obtained, but their expressions are more complicated and are omitted.

Note that the asymptotic formula $I_r(x)\sim\frac{1}{\sqrt{2\pi x}}e^x$ as $x\rightarrow\infty$ (see Olver et al$.$ \cite{o10}) easily allows one to to verify (\ref{eq:var}) in this case.  Writing $\mbox{Var}(X)=\mathbb{E}[X(X-1)]+\mathbb{E}X-(\mathbb{E}X)^2$, the Tur\'{a}n inequality $I_r(x)^2>I_{r+1}(x)I_{r-1}(x)$ (see Amos \cite{a74}) also allows direct verification that $\mbox{Var}(X)<\mathbb{E}X$, which follows from the convex ordering in Proposition \ref{prop:cmp_ord1}.  The total variation bound in that same result may be expressed as
$$
d_{TV}(\mathcal{L}(X),\mbox{Po}(\mathbb{E}X))\leq\sqrt{\lambda}\left(\frac{I_1(2\sqrt{\lambda})}{I_0(2\sqrt{\lambda})}-\frac{I_2(2\sqrt{\lambda})}{I_1(2\sqrt{\lambda})}\right)\,.
$$ 

\section{Convergence and approximation for CMB distributions}\label{sec:stein}

In this section we will use Stein's method for probability approximation to derive an explicit bound on the convergence of $Y\sim\mbox{CMB}(n,\lambda/n^\nu,\nu)$ to $X\sim\mbox{CMP}(\lambda,\nu)$ as $n\rightarrow\infty$.  This convergence is the analogue of the classical convergence of the binomial distribution to a Poisson limit, which corresponds to the case $\nu=1$ here.

Stein's method was first developed by Stein \cite{s72} in the context of normal approximation. The same techniques were applied to Poisson approximation by Chen \cite{c75}.  An account of the method for Poisson approximation, together with a wealth of examples, is given by Barbour, Holst and Janson \cite{bhj92}.  Stein's method has also found a large number of applications beyond the classical normal and Poisson approximation settings.  For an introduction to Stein's method and discussion of its wide applicability, the reader is referred to Barbour and Chen \cite{bc05} and references therein.

For future use, we define
\begin{equation*}
g_\nu(\lambda)= \left\{
\begin{array}{ll}
\min\left\{1,\max\left\{1+\frac{1}{\nu},\left(\frac{3}{2}\right)^\nu\right\}\left(1-\lambda^{-1/2\nu}\right)^{1/\nu-1}\lambda^{1/2\nu-1}\right\} & \mbox{ if $\nu\geq1$ and $\lambda>1$}\,, \\
\left(1+\frac{1}{\nu}\right)\left(1+\lambda^{-1/2\nu}\right)^{1/\nu-1}\lambda^{1/2\nu-1} & \mbox{ if $\nu\leq1$ and $\lambda\geq1$}\,, \\
1 & \mbox{ if $\nu\geq1$ and $\lambda\leq 1$}\,, \\
\left(1-\lambda^{1-\nu}\right)^{-1} & \mbox{ if $\nu<1$ and $\lambda<1$}\,. \\
\end{array} \right.  
\end{equation*}

We use much of the remainder of this section to prove the following.
\begin{theorem}\label{thm:stein}
Let $Y\sim\mbox{CMB}(n,\lambda/n^\nu,\nu)$ for some $0<\lambda<n^\nu$, and let $X\sim\mbox{CMP}(\lambda,\nu)$.  Then
\begin{align*}
d_{TV}(\mathcal{L}(Y),\mathcal{L}(X))&\leq \lambda\left(\frac{\lambda}{n^\nu-\lambda}+\frac{\nu n^\nu\mathbb{E}Y}{n(n^\nu-\lambda)}\right)\left(g_\nu(\lambda)+(1+\mathbb{E}Y)\min\left\{1,\lambda^{-1}\right\}\right)\\
&\quad+\frac{\lambda c_\nu}{n}\left(\mathbb{E}Y+\mathbb{E}Y^2\right)\min\left\{1,\lambda^{-1}\right\}\,,
\end{align*}
where $c_\nu=\max\{1,\nu\}$ and $g_\nu(\lambda)$ is as defined above.
\end{theorem}

\begin{remark}\emph{Taking $\lambda=n^\nu p$ in Theorem \ref{thm:stein} gives the following bound for the total variation distance between the laws of $Y\sim\mbox{CMB}(n,p,\nu)$ and $X\sim\mbox{CMP}(n^\nu p,\nu)$:
\begin{align*}
d_{TV}(\mathcal{L}(Y),\mathcal{L}(X))&\leq n^\nu p\left(\frac{p}{1-p}+\frac{\nu \mathbb{E}Y}{n(1-p)}\right)\left(g_\nu(n^\nu p)+(1+\mathbb{E}Y)\min\left\{1,(n^\nu p)^{-1}\right\}\right)\\
&\quad+n^{\nu-1}pc_\nu\left(\mathbb{E}Y+\mathbb{E}Y^2\right)\min\left\{1,(n^\nu p)^{-1}\right\}\,.
\end{align*}
However, we prefer to work with the parameters given in Theorem \ref{thm:stein}, because in that theorem the CMP distribution (which we regard as the limit distribution) does not depend on $n$.}
\end{remark}

\begin{remark}\label{numrem}\emph{For large $n$, the bound of Theorem \ref{thm:stein} is of order $n^{-\min\{1,\nu\}}$, which is in agreement with the order of upper bound in the classical case $\nu=1$ that is given by Barbour, Holst and Janson \cite{bhj92}.  In fact, Barbour and Hall \cite{bh84} obtained a lower bound of the same order:}
\begin{equation}\label{tvlowerupper}\frac{1}{32}\min\left\{\lambda,\lambda^2\right\}n^{-1}\leq d_{TV}(\mathrm{Bin}(n,\lambda/n),\mathrm{Po}(\lambda))\leq \min\left\{\lambda,\lambda^2\right\}n^{-1}.
\end{equation}

\emph{It would be desirable to obtain a corresponding lower bound for all $\nu>0$, although their method of proof does not generalise easily to $\nu\not=1$.  We have, however, been able to get a good indication of the `true' rate of convergence via a simple numerical study.  We fixed $\lambda=1$ and considered a number of different values of $\nu$.  For each $\nu$, we used Mathematica to evaluate $d_{\nu,n}:=d_{TV}(\mathrm{CMB}(n,1/n^{\nu},\nu),\mathrm{CMP}(1,\nu))$ for different values of $n$.  The values of $\log(d_{\nu,n})$ were plotted against $\log(n)$ and the gradient of the line of best fit was used to estimate the exponent of $n$.  The results of this study strongly suggest that the convergence is indeed of order $n^{-\min\{1,\nu\}}$ (for general $\lambda$, but see Remark \ref{rem3.4} below for a choice of $\lambda$ that gives a faster rate).  For example, in the case $\nu=1/2$ the fitted gradient was $-0.502$, and the fitted gradient was $-0.974$ for $\nu=3/2$.  A direction for future research is to verify this assertion theoretically by obtaining a lower bound of this order.} 
\end{remark}

\begin{remark}\emph{For general $\nu$, we do not have closed-form formulas for the moments $\mathbb{E}Y$ and $\mathbb{E}Y^2$. However, $\mathbb{E}Y^k\approx\mathbb{E}X^k$ for large $n$, and so we can use the asymptotic formula $\mathbb{E}X^k\approx\lambda^{k/\nu}$ to see that, for large $\lambda$, the upper bound of Theorem \ref{thm:stein} is of order}
\begin{equation*}\frac{\lambda^{2/\nu}}{n}+\frac{\lambda^{1/\nu+1}}{n^{\nu}}.
\end{equation*}
\emph{For $\nu=1$, this dependence on $\lambda$ is not as good as the $O(\lambda)$ rate of (\ref{tvlowerupper}).}
\end{remark}

\begin{remark}\label{rem3.4}\emph{In the special case $\lambda=\mathbb{E}Y^\nu$, the rate improves to order $n^{-1}$: 
\begin{equation}d_{TV}(\mathcal{L}(Y),\mathcal{L}(X))\leq\min\big\{1,\lambda\}\frac{c_\nu}{n}\big(\mathbb{E}Y+\mathbb{E}Y^2\big)\,.
\end{equation} 
This bound can be easily read off from the proof of Theorem \ref{thm:stein}.}
\end{remark}

\begin{remark}\emph{As we shall see, the proof of Theorem \ref{thm:stein} relies on a stochastic ordering argument.  Before arriving at this proof, we considered generalising the classical Stein's method proofs of the Poisson approximation of the binomial distribution.  These approaches involve local couplings or size bias couplings (see Barbour, Holst and Janson \cite{bhj92}).  However, neither of these approaches generalise easily to the CMP approximation of the CMB distribution.  The first step in generalising the classical proofs is to write $Y\sim\mbox{CMB}(n,\lambda/n^\nu,\nu)$ as a sum of Bernoulli random variables.  However, these Bernoulli random variables are strongly dependent (see (\ref{eq:ex})), and so local couplings are not applicable.  Also, the natural generalisation of the size-bias coupling approach involves the construction of the power-bias distribution of $Y$, which we found resulted in intractable calculations.}
\end{remark}

The starting point for applying Stein's method is the characterisation of the CMP distribution given by Lemma \ref{lem:char} (i).  Using that, we have the representation
\begin{align}
\nonumber d_{TV}(\mathcal{L}(Y),\mathcal{L}(X))&=\sup_{A\subseteq\mathbb{Z}^+}\left|\lambda\mathbb{E}f_A(Y+1)-\mathbb{E}[Y^\nu f_A(Y)]\right|\\
\label{eq:tv}&=\sup_{A\subseteq\mathbb{Z}^+}\left|\lambda\mathbb{E}f_A(Y+1)-\mathbb{E}Y^\nu\mathbb{E}f_A(Y^{(\nu)})\right|\,,
\end{align}
where $f_A:\mathbb{Z}^+\mapsto\mathbb{R}$ solves the Stein equation
\begin{equation}\label{eq:stein}
I(x\in A)-\mathbb{P}(X\in A)=\lambda f_A(x+1)-x^\nu f_A(x).
\end{equation}
Hence, in proving Theorem \ref{thm:stein} we find a bound on $\left|\lambda\mathbb{E}f_A(Y+1)-(\mathbb{E}Y^\nu)\mathbb{E}f_A(Y^{(\nu)})\right|$ which holds uniformly in $A\subseteq\mathbb{Z}^+$.  In order to do this, we will need bounds on the functions $f_A$ solving (\ref{eq:stein}).  These are given in Lemma \ref{lem:stein} below, whose proof is deferred until Section \ref{subsec:steinfactors}.
\begin{lemma}\label{lem:stein}The unique solution of the $\mbox{CMP}(\lambda,\nu)$ Stein equation (\ref{eq:stein}) is given by $f_A(0)=0$ and, for $j\geq0$,
\begin{equation}\label{eq:steinsolna}f_A(j+1)=\frac{(j!)^\nu}{\lambda^{j+1}}\sum_{k=0}^j\frac{\lambda^k}{(k!)^\nu}[I(k\in A)-\mathbb{P}(X\in A)],
\end{equation}
for $A\subseteq\mathbb{Z}^+$.  The solution satisfies the bounds
\begin{align}
\label{eq:steinfac1}\sup_{A\subseteq\mathbb{Z}^+}\sup_{j\in\mathbb{Z}^+}|f_A(j)|&\leq g_\nu(\lambda)\,,\\
\label{eq:steinfac2}\sup_{A\subseteq\mathbb{Z}^+}\sup_{j\in\mathbb{Z}^+}|f_A(j+1)-f_A(j)|&\leq\lambda^{-1}\left(1-Z(\lambda,\nu)^{-1}\right)\leq\min\left\{1,\lambda^{-1}\right\}\,,
\end{align}
where $g_\nu(\lambda)$ is as defined above.
\end{lemma}

\begin{remark}\emph{The value of $f_A(0)$ is in fact irrelevant, and we follow the usual convention and set it equal to zero (see Barbour, Holst and Janson \cite{bhj92}, p$.$ 6).  Therefore, to be precise, the function $f_A(j)$, as given by (\ref{eq:steinsolna}), is the unique solution of (\ref{eq:stein}) for $j\geq1$.}
\end{remark}

For use in what follows, we define the forward difference operator $\Delta$ by $\Delta f(j)=f(j+1)-f(j)$, and the supremum norm $\lVert\cdot\rVert$ by $\lVert f\rVert=\sup_{j}|f(j)|$ for all $f:\mathbb{Z}^+\mapsto\mathbb{R}$.

Now, we have that
$$
\lambda\mathbb{E}f_A(Y+1)-\mathbb{E}Y^\nu\mathbb{E}f_A(Y^{(\nu)})= \mathbb{E}Y^\nu\left(\mathbb{E}f_A(Y+1)-\mathbb{E}f_A(Y^{(\nu)})\right)+\left(\lambda-\mathbb{E}Y^\nu\right)\mathbb{E}f_A(Y+1)\,.
$$

Recall from Lemma \ref{lem:ord2} that $Y^{(\nu)}\leq_{st}Y+1$.  Hence, we may follow the methods of Daly, Lef\`evre and Utev \cite{dlu12} and obtain 
\begin{align*}
\left|\lambda\mathbb{E}f_A(Y+1)-\mathbb{E}Y^\nu\mathbb{E}f_A(Y^{\nu})\right|&\leq(\mathbb{E}Y^\nu)
\lVert\Delta f_A\rVert\left(1+\mathbb{E}Y-\mathbb{E}Y^{(\nu)}\right)+\lVert f_A\rVert\left|\lambda-\mathbb{E}Y^{\nu}\right|\\
&=\lVert\Delta f_A\rVert\left((1+\mathbb{E}Y)\mathbb{E}Y^\nu-\mathbb{E}Y^{\nu+1}\right)+\lVert f_A\rVert\left|\lambda-\mathbb{E}Y^{\nu}\right|\,,
\end{align*} 
where we used (\ref{eq:powerbias}) to note that $\mathbb{E}Y^{\nu+1}=\mathbb{E}Y^\nu\mathbb{E}Y^{(\nu)}$.  We may then combine the representation (\ref{eq:tv}) with Lemma \ref{lem:stein} to get
\begin{equation}\label{eq:tvb}
d_{TV}(\mathcal{L}(Y),\mathcal{L}(X))\leq\min\left\{1,\lambda^{-1}\right\}\left((1+\mathbb{E}Y)\mathbb{E}Y^\nu-\mathbb{E}Y^{\nu+1}\right) +g_\nu(\lambda)\left|\lambda-\mathbb{E}Y^{\nu}\right|\,.
\end{equation}
To complete the proof of Theorem \ref{thm:stein}, we use Lemmas \ref{lem:stein1} and \ref{lem:stein2} below.  These make use of the characterisation of the CMB distribution.  The idea of combining characterisations of two distributions when using Stein's method has previously been employed by Goldstein and Reinert \cite{gr13} and D\"obler \cite{d13}. 
\begin{lemma}\label{lem:stein1}
Let $Y\sim\mbox{CMB}(n,\lambda/n^\nu,\nu)$.  Then
$$
\mathbb{E}Y^{\nu+1}\geq\lambda\left(1+\mathbb{E}Y-\frac{c_\nu}{n}\left(\mathbb{E}Y+\mathbb{E}Y^2\right)\right)\,,
$$
where $c_\nu=\max\{1,\nu\}$.
\end{lemma}
\begin{proof}
We use the characterisation of the CMB distribution given in Lemma \ref{lem:char} (ii) to note that
\begin{align*}\mathbb{E}Y^{\nu+1}&=\lambda\left(1-\frac{\lambda}{n^{\nu}}\right)^{-1}\mathbb{E}\left[(Y+1)\left(1-\frac{Y}{n}\right)^{\nu}\right] \\
&\geq \lambda\left(1-\frac{\lambda}
{n^{\nu}}\right)^{-1}\mathbb{E}\left[(Y+1)\left(1-\frac{c_\nu Y}{n}\right)\right] \\
&\geq \lambda\mathbb{E}\left[(Y+1)\left(1-\frac{c_\nu Y}{n}\right)\right]\,.
\end{align*} 
\end{proof}
\begin{lemma}\label{lem:stein2}
Let $Y\sim\mbox{CMB}(n,\lambda/n^\nu,\nu)$.  Then
$$
\left|\lambda-\mathbb{E}Y^\nu\right|\leq\lambda\left(\frac{\lambda}{n^\nu-\lambda}+\frac{\nu n^\nu\mathbb{E}Y}{n(n^\nu-\lambda)}\right)\,.
$$
\end{lemma}
\begin{proof}
Let $p=\lambda/n^\nu$.  Using Lemma \ref{lem:char} (ii),
$$
\lambda-\mathbb{E}Y^\nu=\lambda\left(1-\frac{\mathbb{E}\left(1-\frac{Y}{n}\right)^\nu}{1-p}\right)\,.
$$
The result then follows by applying Taylor's theorem to the function $(1-y)^\nu$.
\end{proof}

Substituting the bounds of Lemmas \ref{lem:stein1} and \ref{lem:stein2} into (\ref{eq:tvb}) completes the proof of Theorem \ref{thm:stein}.

\subsection{Remarks on Lemma \ref{lem:stein2}}

We use this section to give some remarks related to Lemma \ref{lem:stein2}.  Firstly, note that the upper bound given in that lemma is of order $n^{-\min\{1,\nu\}}$, which can in fact easily be seen to be the optimal rate.  Using Lemma \ref{lem:stein3} below, we show that a better bound is possible when $\lambda/n^\nu$ is small, although this improved bound will of course still be of the same order as the bound given in Lemma \ref{lem:stein2}.
\begin{lemma}\label{lem:stein3}
Let $Y\sim\mbox{CMB}(n,p,\nu)$ with $n>1$.  There exists $p^\star\in(0,1]$ such that for $p\leq p^\star$
\begin{equation*}
\mathbb{E}Y^\nu \left\{
\begin{array}{ll}
\leq n^\nu p & \mbox{ if $\nu\geq1$}\,, \\
\geq n^\nu p & \mbox{ if $\nu<1$}\,. \\
\end{array} \right.  
\end{equation*}
\end{lemma}
\begin{proof}
Let
$$
h(p)=\mathbb{E}Y^\nu=\frac{\sum_{j=0}^nj^\nu\binom{n}{j}^\nu p^j(1-p)^{n-j}}{\sum_{j=0}^n\binom{n}{j}^\nu p^j(1-p)^{n-j}}\,.
$$
Elementary calculations show that $h(0)=0$, $h^\prime(0)=n^\nu$ and 
$$
h^{\prime\prime}(0)=-2n^\nu\left(n^\nu-(n-1)^\nu-1\right)\,.
$$
Note that (since $n>1$), $h^{\prime\prime}(0)<0$ for $\nu\geq1$ and $h^{\prime\prime}(0)>0$ for $\nu<1$.  Using the continuity of $h$ and Taylor's theorem applied to $h$, the result follows.
\end{proof}

Consider now the case $Y\sim\mbox{CMB}(n,\lambda/n^\nu,\nu)$ with $\nu\geq1$.  By Lemma \ref{lem:stein3}, for $n$ sufficiently large we have that $\lambda-\mathbb{E}Y^\nu\geq0$, and we may then follow the proof of Lemma \ref{lem:stein2} to get the bound
$$
\lambda-\mathbb{E}Y^\nu\leq\lambda\left(\frac{\nu n^\nu\mathbb{E}Y}{n\left(n^\nu-\lambda\right)}-\frac{\lambda}{n^\nu-\lambda}\right)\,,
$$
which improves upon Lemma \ref{lem:stein2}.

A similar argument in the case $\nu<1$ gives that, for $n$ sufficiently large, $\lambda-\mathbb{E}Y^\nu\leq0$ and
$$
\mathbb{E}Y^\nu-\lambda\leq\frac{\lambda^2}{n^\nu-\lambda}\,.
$$

\subsection{Proof of Lemma \ref{lem:stein}}\label{subsec:steinfactors}

It is straightforward to verify that (\ref{eq:steinsolna}), denoted by $f_A(j)$, solves the Stein equation (\ref{eq:stein}).  To establish uniqueness of the solution, we take $j=0$ in (\ref{eq:stein}), from which it follows that any function $h_A(j)$ that solves the Stein equation (\ref{eq:stein}) must satisfy $h_A(1)=f_A(1)$.  By iteration on $\lambda h_A(j+1)-j^\nu h_A(j)=\lambda f_A(j+1)-j^\nu f_A(j)$ it follows that $h_A(j)=f_A(j)$ for all $j\geq 1$, which confirms the uniqueness of the solution.

We now establish (\ref{eq:steinfac2}).  By constructing $X\sim\mbox{CMP}(\lambda,\nu)$ as the equilibrium distribution of a birth-death process with birth rates $\alpha_j=\lambda$ and death rates $\beta_j=j^\nu$, the first inequality of (\ref{eq:steinfac2}) follows from Corollary 2.12 of Brown and Xia \cite{bx01}.  Since $Z(\lambda,\nu)\geq1$ for all $\lambda$ and $\nu$, it follows that $\lambda^{-1}\left(1-Z(\lambda,\nu)^{-1}\right)\leq\lambda^{-1}$.  Finally,
$$
\lambda^{-1}(1-Z(\lambda,\nu)^{-1})=\frac{\sum_{j=0}^{\infty}\frac{\lambda^j}{((j+1)!)^{\nu}}}{\sum_{k=0}^{\infty}\frac{\lambda^k}{(k!)^{\nu}}}\leq 1\,.
$$
This completes the proof of (\ref{eq:steinfac2}).

\begin{remark}\emph{The upper bound $\lambda^{-1}\left(1-Z(\lambda,\nu)^{-1}\right)$ for the forward difference is attained by $f_{\{1\}}(2)-f_{\{1\}}(1)$.}
\end{remark}

It remains to establish (\ref{eq:steinfac1}). We do this by considering separately four cases.  Our strategy is to suitably generalise the proof of Lemma 1.1.1 of Barbour, Holst and Janson \cite{bhj92}, which gives analogous bounds in the Poisson case ($\nu=1$).

Firstly, note that from (\ref{eq:steinfac2}) and the choice $f_A(0)=0$, 
\begin{equation}\label{eq:stein_1}
|f_A(1)|\leq\min\{1,\lambda^{-1}\}\,,
\end{equation}
for each $A\subseteq\mathbb{Z}^+$.  Given (\ref{eq:stein_1}), we need only to show the stated bound on $|f_A(j+1)|$ for $j\geq1$ in each of the four cases detailed below. 

\subsubsection*{Case I: $\nu\geq1$ and $\lambda>1$}

Note that by examining the proof of Lemma 1.1.1 of Barbour, Holst and Janson \cite{bhj92}, it is clear that $|f_A(j+1)|\leq5/4$ for all $j\geq1$ whenever $\nu\geq1$.  We can, however, do a little better.  Barbour and Eagleson \cite{be83}, Lemma 4, obtained the bound $|f_A(j+1)|\leq1$ when $\nu=1$.  By examining their proof we see that the bound also holds for all $\nu\geq1$.

Now, let $U_m=\{0,1,\ldots,m\}$.  It is easily verified that the solution $f_A$ to the Stein equation (\ref{eq:stein}) is given by
\begin{align*}
f_A(j+1)&=\lambda^{-j-1}(j!)^\nu Z(\lambda,\nu)\big(\mbox{CMP}(\lambda,\nu)\{A\cap U_j\}-\mbox{CMP}(\lambda,\nu)\{A\}\mbox{CMP}(\lambda,\nu)\{U_j\}\big)\\
&=\lambda^{-j-1}(j!)^\nu Z(\lambda,\nu)\big(\mbox{CMP}(\lambda,\nu)\{A\cap U_j\}\mbox{CMP}(\lambda,\nu)\{U_j^c\}\\
&\qquad\qquad-\mbox{CMP}(\lambda,\nu)\{A\cap U_j^c\}\mbox{CMP}(\lambda,\nu)\{U_j\}\big)\,,
\end{align*}
where $\mbox{CMP}(\lambda,\nu)\{A\}=\mathbb{P}(X\in A)$.  Hence
\begin{equation}\label{eq:stein_2}
|f_A(j+1)|\leq\lambda^{-j-1}(j!)^\nu Z(\lambda,\nu)\mbox{CMP}(\lambda,\nu)\{U_j\}\mbox{CMP}(\lambda,\nu)\{U_j^c\}\,,
\end{equation}
with equality for $A=U_j$.

Equation (\ref{eq:stein_2}) gives us two ways of bounding $|f_A(j+1)|$.  Firstly, note that
\begin{equation}\label{eq:stein_3}
|f_A(j+1)|\leq\lambda^{-j-1}(j!)^\nu Z(\lambda,\nu)\mbox{CMP}(\lambda,\nu)\{U_j\}=\lambda^{-1}\sum_{r=0}^j\lambda^{-r}\left(\frac{j!}{(j-r)!}\right)^\nu\,,
\end{equation}
and when $j^\nu<\lambda$, this may be bounded to give
\begin{equation}\label{eq:stein_4}
|f_A(j+1)|\leq\lambda^{-1}\sum_{r=0}^j\left(\frac{j^\nu}{\lambda}\right)^r\leq\frac{1}{\lambda-j^\nu}\,.
\end{equation}
Secondly, we also have
\begin{equation}\label{eq:stein_5}
|f_A(j+1)|\leq\lambda^{-j-1}(j!)^\nu Z(\lambda,\nu)\mbox{CMP}(\lambda,\nu)\{U_j^c\}=\lambda^{-1}\sum_{r=j+1}^\infty\lambda^{r-j}\left(\frac{j!}{r!}\right)^\nu\,,
\end{equation}
and when $(j+2)^\nu>\lambda$, this may be bounded to give
\begin{equation}\label{eq:stein_6}
|f_A(j+1)|\leq\frac{1}{(j+1)^\nu}\sum_{r=0}^\infty\left(\frac{\lambda}{(j+2)^\nu}\right)^r=\frac{(j+2)^\nu}{(j+1)^\nu\left((j+2)^\nu-\lambda\right)}\,.
\end{equation}
Note that the bounds (\ref{eq:stein_3})--(\ref{eq:stein_6}) hold for all values of $\nu$ and $\lambda$.  We will also make use of these bounds in the other cases we consider below. 

Now, for $j^\nu\leq\lambda-\lambda^{1-1/2\nu}$, we use (\ref{eq:stein_4}) to get that 
\begin{equation}\label{eq:stein_bd1}
|f_A(j+1)|\leq\lambda^{1/2\nu-1}\,.
\end{equation}  
Similarly, when $(j+2)^\nu\geq\lambda+\lambda^{1-1/2\nu}$, we use (\ref{eq:stein_6}) to get that 
\begin{equation}\label{eq:stein_bd2}
|f_A(j+1)|\leq(3/2)^\nu\lambda^{1/2\nu-1}\,,
\end{equation}
noting that $(j+2)^\nu>j^\nu$.  It remains only to treat the case $|j^\nu-\lambda|<\lambda^{1-1/2\nu}$.

To that end, let $\lambda-\lambda^{1-1/2\nu}<j^\nu<\lambda$, and use (\ref{eq:stein_3}) to note that
$$
|f_A(j+1)|\leq\lambda^{-1}\left(\sum_{r=0}^{\lfloor B\rfloor}a_r+\sum_{r=\lfloor B\rfloor+1}^ja_r\right)\,,
$$
for any $B\leq j$, where 
$$
a_r=\lambda^{-r}\left(\frac{j!}{(j-r)!}\right)^\nu\,.
$$
Note that $|a_r|<1$ for each $r\in\mathbb{Z}^+$.  We choose 
$$
B=\lambda^{1/\nu}\left[1-(1-\lambda^{-1/2\nu})^{1/\nu}\right]\,,
$$
so that $a_{r+1}/a_r<1-\lambda^{-1/2\nu}$ for all $r>B$.  Hence, we have
\begin{equation}\label{eq:stein_7}
|f_A(j+1)|\leq\lambda^{-1}\left(B+1+\frac{1-\lambda^{-1/2\nu}}{\lambda^{-1/2\nu}}\right)=\lambda^{1/\nu-1}\left[1-(1-\lambda^{-1/2\nu})^{1/\nu}\right]+\lambda^{1/2\nu-1}\,.
\end{equation}
Note that, by Taylor's theorem and since $\nu\geq1$, 
$$
1-(1-\lambda^{-1/2\nu})^{1/\nu}\leq\frac{1}{\nu}\left(1-\lambda^{-1/2\nu}\right)^{1/\nu-1}\lambda^{-1/2\nu}\,.
$$ 
Hence,
\begin{equation}\label{eq:stein_bd3}
|f_A(j+1)|\leq\left(1+\frac{1}{\nu}\right)\left(1-\lambda^{-1/2\nu}\right)^{1/\nu-1}\lambda^{1/2\nu-1}\,.
\end{equation}

Finally, we consider the case $\lambda<j^\nu<\lambda+\lambda^{1-1/2\nu}$.  From (\ref{eq:stein_5}) we have
$$
|f_A(j+1)|\leq\lambda^{-1}\left(\sum_{r=j+1}^{\lfloor C\rfloor}b_r+\sum_{r=\lfloor C\rfloor+1}^\infty b_r\right)\,,
$$
for $C\geq j$, where
$$
b_r=\lambda^{r-j}\left(\frac{j!}{r!}\right)^\nu\,.
$$
Analogously to before, we note that $|b_r|<1$ for each $r$, and we make the choice $C=\lambda^{1/\nu}\left(1+\lambda^{-1/2\nu}\right)^{1/\nu}$ so that $b_{r+1}/b_r\leq(1+\lambda^{-1/2\nu})$ for $r>C$.  We then get the bound
\begin{equation}\label{eq:stein_8}
|f_A(j+1)|\leq\lambda^{1/\nu-1}\left[(1+\lambda^{-1/2\nu})^{1/\nu}-1\right]+\lambda^{1/2\nu-1}\,.
\end{equation}
Using Taylor's theorem, 
$$
(1+\lambda^{-1/2\nu})^{1/\nu}-1\leq\frac{1}{\nu}\lambda^{-1/2\nu}\,,
$$
since $\nu\geq1$, and so
\begin{equation}\label{eq:stein_bd4}
|f_A(j+1)|\leq\left(1+\frac{1}{\nu}\right)\lambda^{1/2\nu-1}\,.
\end{equation}

Combining the bounds (\ref{eq:stein_bd1}), (\ref{eq:stein_bd2}), (\ref{eq:stein_bd3}) and (\ref{eq:stein_bd4}) we obtain the stated bound on $\lVert f_A\rVert$ in this case.

\begin{remark}\emph{Recall (\ref{eq:stein_2}).  Taking $j\approx\lambda^{1/\nu}$, and using Stirling's formula and (\ref{eq:norm}), gives
$$
|f_A(j+1)|\approx\frac{\lambda^{1/2\nu-1}}{(2\pi)^{(\nu-1)/2}\sqrt{\nu}}\,,
$$
for $j\approx\lambda^{1/\nu}$ and large $\lambda$.  Hence, a bound of order $\lambda^{1/2\nu-1}$ is the best that we can expect for $\lVert f_A\rVert$ for large $\lambda$.  This order is achieved by Lemma \ref{lem:stein}.  This remark also applies to Case II considered below.}
\end{remark}

\subsubsection*{Case II: $\nu\leq1$ and $\lambda\geq1$}

Here we use an analogous argument to that employed in Case I.  The bounds (\ref{eq:stein_bd1}) and (\ref{eq:stein_bd2}) still apply; the only changes to our argument come for the cases where $|j^\nu-\lambda|<\lambda^{1-1/2\nu}$.

When $\lambda-\lambda^{1-1/2\nu}<j^\nu<\lambda$, we again use (\ref{eq:stein_7}).  Since $\nu\leq1$ in this case, Taylor's theorem gives
$$
1-(1-\lambda^{-1/2\nu})^{1/\nu}\leq\frac{1}{\nu}\lambda^{-1/2\nu}\,,
$$
from which it follows that
$$
|f_A(j+1)|\leq\left(1+\frac{1}{\nu}\right)\lambda^{1/2\nu-1}\,.
$$

When $\lambda<j^\nu<\lambda+\lambda^{1-1/2\nu}$, we use (\ref{eq:stein_8}), noting that, since $\nu<1$,
$$
(1+\lambda^{-1/2\nu})^{1/\nu}-1\leq\frac{1}{\nu}\left(1+\lambda^{-1/2\nu}\right)^{1/\nu-1}\lambda^{-1/2\nu}\,,
$$
giving
$$
|f_A(j+1)|\leq\left(1+\frac{1}{\nu}\right)\left(1+\lambda^{-1/2\nu}\right)^{1/\nu-1}\lambda^{1/2\nu-1}\,.
$$
The stated bound follows.

\subsubsection*{Case III: $\nu\geq1$ and $\lambda\leq1$}

As before, we may use the proof of Lemma 4 of Barbour and Eagleson \cite{be83} to obtain the bound $|f_A(j+1)|\leq1$ for all $\nu\geq1$.

\subsubsection*{Case IV: $\nu<1$ and $\lambda<1$}

Here we again use (\ref{eq:stein_5}).  That bound gives us
$$
|f_A(j+1)|\leq\lambda^{-1}\sum_{r=j+1}^\infty\left(\frac{\lambda^{r-j}}{(r-j)!}\right)^\nu\binom{r}{j}^{-\nu}\left(\lambda^{1-\nu}\right)^{r-j}\leq\lambda^{\nu-1}\sum_{r=1}^\infty\left(\lambda^{1-\nu}\right)^r\,.
$$ 
Since $\lambda<1$ and $\nu<1$, we have $\lambda^{1-\nu}<1$.  Hence we get the bound
$$
|f_A(j+1)|\leq\frac{1}{1-\lambda^{1-\nu}}\,.
$$

\begin{remark}\emph{Consider the case $\nu=0$ and $\lambda<1$.  We then have that our CMP random variable $X$ has a geometric distribution, supported on $\mathbb{Z}^+$, with parameter $\mathbb{P}(X=0)=1-\lambda$.  Note that in this case, Lemma \ref{lem:stein} gives the bound $\lVert f_A\rVert\leq (1-\lambda)^{-1}$.  This was shown, in Remark 4.1 of Daly \cite{d10}, to be the correct dependence on $\lambda$ for such a bound.}
\end{remark}

\section{Other convergence and approximation results}\label{sec:conv}

In this section we consider other convergence and approximation results related to CMP distributions.

\subsection{Sums of Bernoulli random variables}

In Section \ref{sec:stein} we have considered the convergence of the CMB distribution to an appropriate CMP limit.  In this case we were able to derive an explicit bound on this convergence.  Recalling (\ref{eq:ex}), we are able to write a CMB distribution as a sum of Bernoulli random variables (having a particular dependence structure), with each Bernoulli summand having the same marginal distribution.  In this section we consider how we may generalise (\ref{eq:ex}) to a sum of Bernoulli random variables which are no longer exchangeable and yet give a CMP limiting distribution in an analogous way to the limit considered in Section \ref{sec:stein}.  In this case, although we are able to prove convergence in distribution, we are unable to give an explicit bound on the convergence rate; further discussion is given in Remark \ref{remark sect4}.

Consider the following generalisation of (\ref{eq:ex}).  Let $X_1,\ldots,X_n$ be Bernoulli random variables with joint distribution given by
$$
\mathbb{P}(X_1=x_1,\ldots,X_n=x_n)=\frac{1}{C_n'}\binom{n}{k}^{\nu-1} \prod_{j=1}^np_j^{x_j}(1-p_j)^{1-x_j}\,,
$$
where $k=x_1+\cdots+x_n$ and the normalizing constant $C_n^\prime$ is given by
$$
C_n'=\sum_{k=0}^n\binom{n}{k}^{\nu-1}\sum_{A\in F_k}\prod_{i\in A}p_i\prod_{j\in A^c}(1-p_j)\,,
$$
where
$$
F_k=\left\{A\subseteq\{1,\ldots,n\}:|A|=k\right\}\,.
$$
We consider the convergence of the sum $W=X_1+\cdots+X_n$.  It is easy to see that $W$ has mass function
\begin{equation}\label{eq:cmppb}
\mathbb{P}(W=k)=p_{n,\nu}(k;p_1,\ldots,p_n)=\frac{1}{C_n'}\binom{n}{k}^{\nu-1}\sum_{A\in F_k}\prod_{i\in A}p_i\prod_{j\in A^c}(1-p_j)\,,
\end{equation}
for $k=0,1,\ldots,n$.  This distribution generalises the Poisson binomial distribution in a way analogous to the CMP and CMB generalisations of the Poisson and binomial distributions.   We therefore say that a random variable with mass function (\ref{eq:cmppb}) follows the Conway-Maxwell-Poisson binomial (CMPB) distribution.  Of course, the case $\nu=1$ is the usual Poisson binomial distribution and the case $p_1=\cdots=p_n=p$ reduces to the $\mbox{CMB}(n,p,\nu)$ distribution.
\begin{theorem}\label{thm:cmppb}
Let $W=X_1+\cdots+X_n$ be as above, with mass function $p_{n,\nu}(k;p_1,\ldots,p_n)$ given by (\ref{eq:cmppb}) with $p_i=\frac{\lambda_i}{n^\nu}$ for $i=1,\ldots,n$, where the $\lambda_i$ are positive constants that do not involve $n$.  Then $W$ converges in distribution to $X\sim\mbox{CMP}(\lambda,\nu)$ as $n\rightarrow\infty$, where
$$
\lambda=\lim_{n\rightarrow\infty}\frac{1}{n}\sum_{i=1}^n\lambda_i\,.
$$
\end{theorem}
\begin{proof}
Firstly, note that in the case $\nu=1$ the result is known.  It is the classical convergence of a sum of independent Bernoulli random variables to a Poisson distribution.  That immediately gives us the limit
\begin{equation}\label{eq:lim1}
\frac{1}{n^k}\sum_{A\in F_k}\prod_{i\in A}\lambda_i\prod_{j\in A^c}\bigg(1-\frac{\lambda_j}{n}\bigg)\rightarrow e^{-\lambda}\frac{\lambda^k}{k!}\,,
\end{equation} 
as $n\rightarrow\infty$.  As a consequence of (\ref{eq:lim1}), we have that
\begin{equation}\label{eq:lim2}
\frac{1}{n^k}\sum_{A\in F_k}\prod_{i\in A}\lambda_i\rightarrow \frac{\lambda^k}{k!}\,,
\end{equation}
as $n\rightarrow\infty$, since 
$$
\lim_{n\rightarrow\infty}\prod_{j\in A^c}\left(1-\frac{\lambda_j}{n}\right)=\lim_{n\rightarrow\infty}\prod_{j=1}^n\left(1-\frac{\lambda_j}{n}\right)\cdot\lim_{n\rightarrow\infty}\prod_{l\in A}\left(1-\frac{\lambda_l}{n}\right)^{-1}=e^{-\lambda}\,.
$$

Now, in the present case we may write the mass function (\ref{eq:cmppb}) in the form
$$
p_{n,\nu}(k;\lambda_1/n^\nu,\ldots,\lambda_n/n^\nu)=\frac{1}{C_n'}\left(\frac{n!}{(n-k)!n^k}\right)^{\nu-1}\frac{1}{(k!)^{\nu-1}}\sum_{A\in F_k}\prod_{i\in A}\frac{\lambda_i}{n}\prod_{j\in A^c}\left(1-\frac{\lambda_j}{n^{\nu}}\right)\,.
$$
Clearly
$$
\lim_{n\rightarrow\infty}\frac{n!}{(n-k)!n^k}=1\,, 
$$
and
$$
\prod_{j\in A^c}\left(1-\frac{\lambda_j}{n^{\nu}}\right)=\prod_{j=1}^n\left(1-\frac{\lambda_j}{n^{\nu}}\right)\prod_{l\in A}\left(1-\frac{\lambda_l}{n^{\nu}}\right)^{-1}\,.
$$
Note that
$$
\lim_{n\rightarrow\infty}\prod_{l\in A}\bigg(1-\frac{\lambda_l}{n^{\nu}}\bigg)^{-1}=1\,,
$$
and that the product
$$
\prod_{j=1}^n\bigg(1-\frac{\lambda_j}{n^{\nu}}\bigg)\,
$$
and the normalizing constant $C_n^\prime$ do not depend on $k$.  Combining these observations with (\ref{eq:lim2}), we have that
$$
\lim_{n\rightarrow\infty}p_{n,\nu}(k;\lambda_1/n^\nu,\ldots,\lambda_n/n^\nu)=\frac{C\lambda^k}{(k!)^\nu}\,,
$$
where $C$ does not depend on $k$.  The result follows.
\end{proof}

\begin{remark}\label{remark sect4} \emph{It would be desirable to extend Theorem \ref{thm:cmppb} to include an explicit bound on the convergence rate, as was achieved in Theorem \ref{thm:stein}.  Such a bound could, in principle, be established by generalising the proof of that theorem.  This approach would require one to obtain a Stein equation for the CMPB distribution, a generalisation of the stochastic ordering result of Lemma \ref{lem:ord2} to the CMPB distribution, and an appropriate extension of the moment estimates of Lemmas \ref{lem:stein1} and \ref{lem:stein2}.   This is a possible direction for future research.}
\end{remark}

\subsection{Mixed CMP distributions}

Finally, we also consider the case of a mixed CMP distribution.  For a non-negative,
real-valued random variable $\xi$, we say that $W\sim\mbox{CMP}(\xi,\nu)$ has a
mixed CMP distribution if
$$
\mathbb{P}(W=j)=\frac{1}{(j!)^\nu}\mathbb{E}\left[\frac{\xi^j}{Z(\xi,\nu)}\right], \quad j\in\mathbb{Z}^+\,.
$$   
We assume throughout that $\xi$ is such that this expectation exists.

Following the proof of Theorem 1.C (for mixed Poisson approximation) in the book by
Barbour, Holst and Janson \cite{bhj92}, we use the characterisation in Lemma
\ref{lem:char} (i), along with the bounds on the solution to the Stein equation
given in Lemma \ref{lem:stein}, to obtain the following.
\begin{theorem}
Let $\xi$ be a non-negative random variable.  Then
$$
d_{TV}(\mbox{CMP}(\xi,\nu),\mbox{CMP}(\lambda,\nu))\leq
g_\nu(\lambda)\mathbb{E}|\xi-\lambda|\,, 
$$
where $g_\nu(\lambda)$ is as defined in Section \ref{sec:stein}.
\end{theorem}

\subsection*{Acknowledgements}
FD is grateful to James Cruise for helpful discussions, and RG would like to thank Steven Gillispie for a helpful discussion on the asymptotics of the CMP normalising constant.  Both authors would like to thank Sebastian Vollmer for assistance with the numerical work of Remark \ref{numrem}.  RG is supported by EPSRC research grant EP/K032402/1.  RG is also grateful to Heriot-Watt University and EPSRC for funding a visit to Heriot-Watt, where many of the details of this project were worked out.  Finally, we would like to thank the referees for their comments and suggestions.

\end{document}